\documentclass[12pt]{iopart}

\usepackage{iopams}

\usepackage{color}

\usepackage[all,knot,arc,poly]{xy}

\newtheorem{theorem}{Theorem}
\newtheorem{prop}{Proposition}

\newtheorem{lemma}{Lemma}

\newtheorem{remark}{Remark}
\newtheorem{example}{Example}

\def\ds{\displaystyle}

\begin{document}

%%-------------------------------------------------------- title and authors

\title[Cubic polynomials on Lie groups: reduction of the Hamiltonian system]{Cubic polynomials on Lie groups: reduction of the Hamiltonian system}

\author{L Abrunheiro$^1$, M Camarinha$^2$ and J Clemente-Gallardo$^3$}

\address{$^1$ ISCA, Universidade de Aveiro, Rua Associa\c c\~ao
  Humanit\'aria dos Bombeiros Volunt\'arios  de Aveiro, 3810-500 Aveiro--Portugal}

\address{$^2$ Departamento de Matem\'atica, Universidade de Coimbra, 3001-501 Coimbra--Portugal}

\address{$^3$ BIFI-Departamento de F\'{\i}sica Te\'orica and Unidad
  Asociada IQFR-BIFI, Universidad de Zaragoza, Edificio I+D, Campus R\'{\i}o Ebro,  C/ Mariano Esquillor s/n,  E-50018 Zaragoza--Spain}

\eads{\mailto{abrunheiroligia@ua.pt},  \mailto{mmlsc@mat.uc.pt}, \mailto{jesus.clementegallardo@bifi.es}}

%%-------------------------------------------------------- abstract              % should not normally exceed 200 words
\begin{abstract}
This paper analyzes the optimal control problem of cubic polynomials on compact Lie
groups from a Hamiltonian point of view and its symmetries. The dynamics of the problem is described by a presymplectic
formalism associated with the canonical symplectic form on the
cotangent bundle of the semidirect product of the Lie group and its
Lie algebra. Using these  control geometric tools, the relation
between the Hamiltonian approach developed here and the known
variational one is analyzed. After making explicit the left
trivialized system, we use the technique of Marsden-Weinstein
reduction to remove the symmetries of the Hamiltonian system. In view
of the reduced dynamics, we are able  to guarantee, by means of the
Lie-Cartan theorem, the existence of a considerable number of
independent integrals of motion in involution.
\end{abstract}
%%--------------------------------------------------------  MSC  ---- Submitto ---- maketitle

% ver melhor   MSC 2010
%\submitto{\JPA} % aparece o antigo nome do jornal e não o actual !!!!!
\maketitle

%%----------------------------------------------------------------------------------------------------------------------------------------%%
%%-------------------------------------------------------- introduction ------------------------------------------------------------------%%
%%----------------------------------------------------------------------------------------------------------------------------------------%%
\section{Introduction}

Riemannian cubic polynomials (RCP), also called Riemannian cubics,  can be seen as a generalization of cubic polynomials in Euclidean spaces to Riemannian manifolds. The cubic polynomials on a Riemannian manifold are the smooth solutions of the fourth order differential equation
\begin{equation} \label{eq:cubic}
\displaystyle{\frac{\mathrm{D}^4x}{\mathrm{d}t^4}+\mathrm{R}\left(\frac{\mathrm{D}^2x}{\mathrm{d}t^2},\frac{\mathrm{d}x}{\mathrm{d}t}\right)\frac{\mathrm{d}x}{\mathrm{d}t}=0},
\end{equation}
where $\mathrm{D}/\mathrm{d}t$ denotes the covariant differentiation   and $\mathrm{R}$ the curvature tensor. The equation (\ref{eq:cubic}) is the Euler-Lagrange equation of a second order variational problem with Lagrangian given by $\frac{1}{2}\left\langle\mathrm{D}^2x/\mathrm{d}t^2,\mathrm{D}^2x/\mathrm{d}t^2\right\rangle$, where $\langle \cdot, \cdot \rangle $ denotes the Riemannian metric. This variational problem was first introduced  in 1989 (see \cite{NoaHeiPad:1989}) and explored from a dynamical interpolation perspective in 1995 (see \cite{CrSLei:1995}). Interesting points related to this subject have been developed in the last few years, namely a geometric theory surprisingly close to the Riemannian theory of geodesics (see \cite{2005AbrCamEquadiff,AbrCam:2005,AbrCamCGal:2007,2003Bloch,CamCrSLei:1995,CamCrSLei:2000,2001CamarinhaandCrouchandSilvaLeite,2002GiamboPicc,Noa:2003,Noa:2006,Pop:2007,2000SilvaLeiteCamCrouch}). We recall, in particular, a result which says that if $V$ denotes the velocity vector field of a cubic polynomial $x$, then
\begin{equation}\label{eq:I1} \displaystyle{I_1=\frac{1}{2}\;\left\langle\frac{\mathrm{D}V}{\mathrm{d}t}\:,\:\frac{\mathrm{D}V}{\mathrm{d}t}\right\rangle-
\left\langle\frac{\mathrm{D}^2V}{\mathrm{d}t^2}\:,\:V\right\rangle}
\end{equation}
is invariant along $x$. In Riemannian context,  $I_1$ plays a role similar to the one played by the length of the velocity vector field in the theory of geodesics (see, for example \cite{2001CamarinhaandCrouchandSilvaLeite}). Recently, in \cite{AbrCam:2005,Noa:2003,Noa:2006,Pop:2007}, the analysis of RCP from a variational point of view was carried out for locally symmetric manifolds and a second invariant was obtained:
\begin{equation}\label{eq:I2} \displaystyle{I_2=\left\langle\frac{\mathrm{D}^2V}{\mathrm{d}t^2},\frac{\mathrm{D}^2V}{\mathrm{d}t^2}\right\rangle-\left\langle\frac{\mathrm{D}^3V}{\mathrm{d}t^3},\frac{\mathrm{D}V}{\mathrm{d}t}\right\rangle}.
\end{equation}
The analysis of RCP given in  \cite{AbrCam:2005,Noa:2003} is
qualitative, with special attention to the case of the Lie group
$SO(3)$, where RCP correspond to Lie quadratics on the Lie
algebra. The article \cite{AbrCam:2005} introduces a reduction of the
RCP's equation for this Lie group of rotations. In \cite{Noa:2003}
some results on asymptotics and symmetries of cubics are proved for
the particular case of the so-called null cubics  on $SO(3)$. In
\cite{Noa:2006}, the author solves by quadratures the linking
equation on $SO(3)$ and $SO(1,2)$ of the Riemannian cubics. Finally,
\cite{Pop:2007} studies n-th order generalizations of RCP introduced
in \cite{CamCrSLei:1995}.

To our knowledge, the first Hamiltonian description of the RCP problem
has been considered in \cite{CamCrSLei:2000} (made in collaboration
with one of the authors). The present  paper deals, for the case of
arbitrary compact and connected Lie groups, with  a different
Hamiltonian description of the problem. Here we use a presymplectic
approach to the Pontryagin's Maximum Principle
inspired by some ideas of
\cite{BarMDieLec:2007,BenMDie:2005,DTelIb:2003,LeoCorMart:2004}.
Namely, we consider the intrinsic geometric approach  used in
\cite{DTelIb:2003,LeoCorMart:2004} for a first order general optimal
control problem, and similarly considered in \cite{BarMDieLec:2007}
for time-dependent optimal control problems by using the jet bundles
framework. In a similar way,  \cite{BenMDie:2005} gives the
geometric treatment of the Lagrangian dynamics with higher-order
constraints. The description of RCP (on an arbitrary manifold) using
these geometric ideas were first presented in \cite{AbrCamCGal:2007}
by the authors of the present paper. Recently, in
\cite{AbrCamCGal:2010,AbrCamCGal:2011},  the authors have treated the particular
situation  of the dynamic control of the spherical free rigid body, a
mechanical system with configuration manifold given by the Lie group
$SO(3)$. The new contribution of this presymplectic formalism is to
use the Lie group structure of the semidirect product of the Lie group
and its Lie algebra, $G\times\mathfrak{g}$, which  is the state space
of the optimal control problem. This allows us to use  classical
results from  \cite{1978AbrahamMarsden,BatesCushman:1997}, adapted to
this Lie group structure. Namely, we present the left trivialization
of the Hamiltonian system, a set of equations which lives in the
manifold given by the cartesian product of the semidirect product
above mentioned and the dual of its Lie algebra,
$G\times\mathfrak{g}\times\mathfrak{g}^*\times\mathfrak{g}^*$.

The main goal of this work is to reduce the degrees of freedom of the left trivialized Hamiltonian system. We first apply the symplectic point reduction theorem (\cite{MarWei:1974,2004OrtegaRatiu}) and then explore the reduced dynamics using a suitable symplectomorphism. The reduced Hamiltonian vector field lives in the manifold given by the cartesian product of a coadjoint orbit, the Lie algebra and its dual, $\mathcal{O}_{\eta}\times\mathfrak{g}\times\mathfrak{g}^*$. Furthermore, some invariants along the extremal trajectories are characterized as a crucial point to develop, in a future work, a study of the integrability of the Hamiltonian system. In fact, using the Lie-Cartan theorem (\cite{1988ArnKozNein}),  we obtain an interesting result on the number of independent integrals of motion in involution.

\nosections

The plan of the paper is as follows. Section 2 recalls some notes on compact Lie
groups and fixes the notation used in the rest of the paper. Section 3 begins with the introduction of the optimal control problem of cubic polynomials and presents the corresponding presymplectic approach. After that, we provide the left trivialized description equivalent to the variational one (\cite{CrSLei:1995}) in a similar way to what happens in \cite{CamCrSLei:2000}. However, it is important to remark that our Hamiltonian system and the one in \cite{CamCrSLei:2000} are different. The last section contains the reduction of the Hamiltonian problem by means of the Marsden-Weinstein technique. As we have mentioned above, the system will be reduced to a system on the cartesian product of a coadjoint orbit, the algebra of the Lie group and the dual of this algebra. In this context, besides the geometric deduction of the two known invariants, we find some more invariants along the Riemannian cubics and find a relevant result on its independence and involution.

%%----------------------------------------------------------------------------------------------------------------------------------------%%
%%-------------------------------------------------------- preliminary notes -------------------------------------------------------------%%
%%----------------------------------------------------------------------------------------------------------------------------------------%%
\section{Preliminary notes}

The present section gives some notations, definitions and results from the Lie groups theory,  which we will be using throughout the paper.

%%-------------------------------------------------------- The Lie group $G$
\subsection{The Lie group $G$} \label{subsection:Liegroup}

Let $G$  be a connected and compact Lie group with identity $e$. The corresponding Lie algebra is $(\mathfrak{g},[.,.])$, where $[.,.]$ is the bracket operation and $\mathfrak{g}^*$ denotes the dual space of this algebra. Furthermore, the elements of $G$ are denoted by $x$ or $g$ and the maps $G\times G\to G, (x,g)\mapsto  xg$ and $G\to G, x\mapsto x^{-1}$ are the multiplication and inversion operations for the Lie group $G$, respectively. Given $x,g\in G$, let  $L_x:G\to G$ and $R_x:G\to G$ be, respectively, the left and right translations by $x$. The tangent of $L_x$ at $g$ is denoted by  $T_gL_x$ and $T^*_gL_x$ represents its transpose. Recall the following definitions:
\begin{itemize}
\item[--] The adjoint representation of the Lie group $G$ is denoted by $\mathrm{Ad}$. It gives for each $x\in G$ an algebra automorphism defined by $\mathrm{Ad}_x=T_e(R_{x^{-1}}\circ L_x).$

\item[--] The adjoint representation of the Lie algebra $\mathfrak{g}$ is the tangent of $\mathrm{Ad}$ at the identity $e$ and it is denoted by $\mathrm{ad}$. For each $Y,Z\in \mathfrak{g}$, we have $\mathrm{ad}_YZ=[Y,Z].$

\item[--] The map $\mathrm{Ad}^*: G\rightarrow Aut(\mathfrak{g}^*)$ defined, by $[\mathrm{Ad}^*(x)](\xi):=\mathrm{Ad}^*_{x^{-1}}\xi=\xi\circ\mathrm{Ad}_{x^{-1}}$, for each $x\in G$ and $\xi\in\mathfrak{g}^*$,  is called the  co-adjoint representation of $G$.

\item[--] The co-adjoint representation of $\mathfrak{g}$ is the map $ \mathrm{ad}^*:\mathfrak{g}\rightarrow Aut(\mathfrak{g}^*)$ defined, for each $Y\in\mathfrak{g}$ and $\xi\in\mathfrak{g}^*$, by $[\mathrm{ad}^*(Y)](\xi):=-\mathrm{ad}^*_Y\xi=-\xi\circ \mathrm{ad}_Y.$
\end{itemize}

\nosections

Since the Lie group is assumed to be connected and compact, we can guarantee the existence of a bi-invariant metric on $G$, which we shall denote by $\langle .,. \rangle$. This statement and the following  result can be found for instance in \cite{KobaNomi:196369}.
\begin{theorem} \cite{KobaNomi:196369} \label{theorem:basicpropertiesLieGroup}
If $G$ is a Lie group equipped with a bi-invariant metric, the metric connection $\nabla$ and the curvature tensor $\mathrm{R}$ associated with that metric are given, respectively, by  $\nabla_YZ=\frac{1}{2}[Y,Z]$ and $\mathrm{R}(Y,Z)W=-\frac{1}{4}[[Y,Z],W]$, where $Y, Z$ and $W$ are left  invariant vector fields. Furthermore, the first above equality  implies that $\left\langle [Y,Z],W \right\rangle =\left\langle Y,[Z,W]\right\rangle$.
\end{theorem}

\nosections

In the course of this paper, we shall fix an orthonormal  basis in the Lie algebra $\mathfrak{g}$. The corresponding dual basis is a basis of the dual space $\mathfrak{g}^*$. These two bases generate left invariant frame and co-frame fields on $G$, respectively. We assume the following notations:
\begin{itemize}
\item[--] Let $Y$ be  a curve in $\mathfrak{g}$ and $\xi$ a curve in $\mathfrak{g}^*$. We represent by $\dot{Y}$ (respectively, $\dot{\xi}$) the element of $\mathfrak{g}$ (respectively, $\mathfrak{g}^*$)  which has components with respect to the basis of $\mathfrak{g}$ (respectively, $\mathfrak{g}^*$) above mentioned, given by the derivative of the components of $Y$  (respectively, $\xi$).
\item[--] Given $\xi\in\mathfrak{g}^*$, the tangent vector identified with this co-vector  by the Riemannian metric will be denoted by $X_{\xi}\in\mathfrak{g}$. That is, $\xi(Y)=\langle X_{\xi},Y\rangle$, $\forall Y\in\mathfrak{g}$.
\item[--] With the above notation, it is simple to verify that $\dot{X}_\xi=X_{\dot{\xi}}$ and $X_{\mathrm{ad}^*_Y\xi}=-\mathrm{ad}_YX_\xi$.
\end{itemize}

%%-------------------------------------------------------- The tangent bundle of $G$ and its left trivialization
\subsection{The tangent bundle Lie group and its left trivialization} \label{subsection:TangentBundleTG}

\begin{lemma} \cite{1993KolarMichorSlovak}
The tangent bundle $TG$ is a Lie group with a group operation defined as the tangent prolongation of the original one on $G$. That is, the multiplication operation for $TG$ is defined by
\[
(v_x,v_g)\in T_xG\times T_gG\longmapsto v_xv_g=T_xR_gv_x+T_gL_xv_g\;\in T_{xg}G
\]
and the inversion is defined by
\[
v_x\in T_xG\longmapsto v_x^{-1}=-\left(T_eL_{x^{-1}}\circ T_xR_{x^{-1}}\right)(v_x) \,\in T_{x^{-1}}G.
\]
\end{lemma}

Consider the semidirect product $G\times\mathfrak{g}$ of the Lie group $G$ and the Lie algebra $\mathfrak{g}$ regarded as abelian group, under  the right action of  $G$ on $\mathfrak{g}$, $(x,Y)\in G\times\mathfrak{g} \longmapsto Ad_{x^{-1}}Y$.
\begin{lemma} \label{lemma:SemidirectProduct}
The semidirect product $G\times\mathfrak{g}$ is a Lie group whose underlying manifold is the cartesian product $G\times\mathfrak{g}$  and group multiplication law
\[
(x,Y)(g,Z)=(xg,\mathrm{Ad}_{g^{-1}}Y+Z),
\]
for $(x,Y),(g,Z)\in G\times\mathfrak{g}$. The inversion is defined as $(x,Y)^{-1}=(x^{-1},-\mathrm{Ad}_xY)$.
\end{lemma}

The semidirect product structure considered here is a special case of the general one defined by a right representation of a Lie group on a vector space  that may be found in works on semidirect products, particulary the ones on models of continuum mechanics and plasmas where it is convenient to work with right instead of left representations (see, for example \cite{HolmMarsRatiu:1986}).

\begin{prop} \cite{1993KolarMichorSlovak}
The left trivialization of  $TG$  determined by the map
\begin{equation} \label{eq:LeftTrivTG}
\begin{array}{llll}\lambda : & TG & \longrightarrow & G\times\mathfrak{g} \\[5pt] & v_x\in T_xG & \longmapsto & \left(x,T_xL_{x^{-1}}v_x\right),\end{array}
\end{equation}
allows us to write   the Lie group diffeomorphism $TG\simeq G\times\mathfrak{g}$.
\end{prop}

We introduce now some important  notations used in the rest of the paper.
\begin{itemize}
\item[--] The elements of the tangent bundle $T(G\times\mathfrak{g})$ are denoted by
   \[
   (v_x,Y,U)\in T_{(x,Y)}(G\times\mathfrak{g})=T_xG\times\{Y\}\times\mathfrak{g}.
   \]

\item[--] The second tangent bundle  of $G$, $T^2G$, can also  be trivialized by using the map $\lambda$ and then realized as a  bundle over $G\times\mathfrak{g}$, which is a  subbundle of $T(G\times\mathfrak{g})$. We represent this bundle by $\overline{T^2}G$ and denote its elements as
    \[
    (v_x,U)\in T_{(x,T_x{L}_{x^{-1}}v_x)}(G\times\mathfrak{g})\simeq T_xG\times\mathfrak{g}.
    \]

\item[--] The elements of  cotangent bundle $T^*(G\times\mathfrak{g})$ are represented by
    \[
    (\alpha_x,Y,\xi)\in T_{(x,Y)}^*(G\times\mathfrak{g})=T^*_xG\times\{Y\}\times\mathfrak{g}^*.
    \]
\end{itemize}
In the previous statements we are considering  $(x,Y)\in G\times\mathfrak{g}$. Throughout this paper we will, for the sake of simplicity, occasionally assume the identification between elements of $T_{(x,Y)}(G\times\mathfrak{g})$ (respectively, $T^*_{(x,Y)}(G\times\mathfrak{g})$) and elements of $T_xG\times\mathfrak{g}$ (respectively, $T^*_xG\times\mathfrak{g}^*$).

According  the  Lie group structure chosen in lemma \ref{lemma:SemidirectProduct}, we easily compute:
\begin{equation} \label{eq:LeftTranslation}
T_{(e,0)}L_{(x,Y)}(Z,U)=(T_eL_xZ,U+\mathrm{ad}_YZ)
\end{equation}
and
\begin{equation}
 \label{eq:adjunt}
\mathrm{ad}_{(Y,Z)}(Y',Z')=\left(\mathrm{ad}_YY',\mathrm{ad}_YZ'+\mathrm{ad}_ZY'\right),
\end{equation} 
where $x\in G$ and $Y, Z, U, Y', Z'\in \mathfrak{g}$. Obviously, these
formulas can be  derived from the general known ones from the theory
of semidirect products.

%%----------------------------------------------------------------------------------------------------------------------------------------%%
%%-------------------------------------------------------- Hamiltonian system ------------------------------------------------------------%%
%%----------------------------------------------------------------------------------------------------------------------------------------%%
\section{Hamiltonian system} \label{section:HamilSystem}

The aim of this section  is to give a Hamiltonian description of the optimal control problem of cubic polynomials on $G$ based on some material published in \cite{AbrCamCGal:2007}, where we used a geometric formulation similar to the one developed in \cite{BenMDie:2005} for higher order constrained variational problems. The section begins with the introduction of the optimal control problem, where by means of the left translation on $G$,  the state space has been left trivialized to be $G\times\mathfrak{g}$ instead of $TG$.  After that, we apply a presymplectic constraint algorithm and the result is a Hamiltonian system on a space symplectomorphic to $T^*(G\times\mathfrak{g})$.  Using again  a left trivialization, but now determined by left translation on the group $G\times\mathfrak{g}$, we pass to a Hamiltonian description on $G\times\mathfrak{g}\times\mathfrak{g}^*\times\mathfrak{g}^*$.

%%-------------------------------------------------------- Optimal control problem
\subsection{Optimal control problem}

Considering  the left trivialization (\ref{eq:LeftTrivTG}) of  $TG$, the state space for our problem  may be taken to be the  semidirect product $G\times\mathfrak{g}$  and the  bundle of controls as the second tangent bundle $\overline{T^2}G$. The optimal control problem of cubic polynomials on $G$ consists in finding the $C^2$ piecewise smooth curve $\gamma:[0,T]\to\overline{T^2}G$ with fixed endpoints in state space, minimizing the functional $\int_0^T L(\gamma(t))\mathrm{d}t$, with $T\in\mathbb{R}^+$ fixed, for $L:\overline{T^2}G\to \mathbb{R}$ the cost functional defined by
\begin{equation}\label{eq:Lagrangian}
L(v_x,U)=\frac{1}{2} \langle U,U \rangle
\end{equation}
and satisfying the control system
\begin{equation} \label{eq:controlsystem}
\frac{\mathrm{d}}{\mathrm{d}t}\left(\tau_2^1(\gamma(t)\right)=F(\gamma(t)),
\end{equation}
where $\tau_2^1:\overline{T^2}G\to G\times\mathfrak{g}$ is the natural
projection and   $F:\overline{T^2}G\to T(G\times\mathfrak{g})$ is the
vector field  along this projection  defined by 
\begin{equation}\label{eq:VFControl}
F(v_x,U)=(v_x,T_xL_{x^{-1}}v_x,U).
\end{equation}

\[
\xymatrix{  & & \overline{T^2}G  \ar[lldd]_{\tau_2^1} \ar[rr]^{F} & & T(G\times\mathfrak{g})
 \\
\\
 G\times\mathfrak{g}  & & [0,T] \ar[ll]^{\tau_2^1\circ\gamma} \ar[uu]^{\gamma} \ar[rruu]_{\frac{\mathrm{d}}{\mathrm{d}t}\left(\tau_2^1\circ\gamma\right)} & &  \\}
\]

Notice that, according the notation set in the subsection
\ref{subsection:TangentBundleTG}, a curve $\gamma$ in
$\overline{T^2}G$ is defined by means of three elements: a curve $x$
in $G$; a vector field $Y_x$ along $x$ (which can be seen as a curve
in $TG$ satisfying $\pi_G\circ Y_x=x$, where $\pi_G:TG\to G$ is the
canonical projection); and a curve $U$ in $\mathfrak{g}$. So,
$\gamma(t)=(Y_x(t),U(t))\in T_{x(t)}G\times\mathfrak{g}$ and we have
$\tau_2^1(\gamma(t))=(x(t),T_{x(t)}L_{x(t)}Y_x(t))$. Consequently,
using the appropriated basis of left invariant vector fields on $G$ to
develop the calculus, it is simply to prove that the control system
(\ref{eq:controlsystem}) can be written as 
\begin{equation}
\dot{x}(t)=Y_x(t)\qquad\frac{\mathrm{D}Y_x}{\mathrm{d}t}(t)=T_eL_{x(t)}U(t),
\end{equation}
which is a  version of the control system presented in papers \cite{BlCr:1996,CamCrSLei:2000}.

%%-------------------------------------------------------- Dynamics of the optimal control problem
\subsection{Dynamics of the optimal control problem}

The co-state space of our system is the cotangent bundle $T^*(G\times\mathfrak{g})$. The dynamics of the control problem is described by a presymplectic system $\left(\mathcal{T},\overline{\Omega},\overline{H}\right)$ whose total space is the bundle over $G\times\mathfrak{g}$ given by
\begin{equation}
\mathcal{T}=T^*(G\times\mathfrak{g})\times_{_{G\times\mathfrak{g}}}\overline{T^2}G.
\end{equation}
The elements of this space are  points in $T^*_xG\times\{Y\}\times\mathfrak{g}^*\times\mathfrak{g}$ denoted by $(\alpha_x,Y,\xi,U)$, where $(x,Y)\in G\times\mathfrak{g}$. Consider the  canonical projections  $pr_1:\mathcal{T}\to T^*(G\times\mathfrak{g})$, $(\alpha_x,Y,\xi,U)\mapsto (\alpha_x,Y,\xi)$ and $pr_2:\mathcal{T}\to\overline{T^2}G$, $(\alpha_x,Y,\xi,U)\mapsto (T_eL_xY,U)$. The closed two form is defined by the  pull-back
\begin{equation} \label{eq:twoform}
\overline{\Omega}=(pr_1)^*\Omega_1,
\end{equation}
with $\Omega_1$ denoting  the canonical symplectic two form on the space $T^*(G\times\mathfrak{g})$. The Hamiltonian is defined by $\overline{H}=\ll pr_1,F\circ pr_2\gg -L\circ pr_2$, where $F$ and $L$ are defined by (\ref{eq:Lagrangian}) and (\ref{eq:VFControl})  and $\ll .,. \gg$ stands for the canonical duality product of vectors and covectors on $G\times\mathfrak{g}$. Then,
\begin{equation} \label{eq:HamiltonianPresympl}
\overline{H}(\alpha_x,Y,\xi,U)=\left(T^*_eL_x\alpha_x\right)(Y)+\xi\left(U\right)-\frac{1}{2}\langle U,U\rangle.
\end{equation}
The dynamical vector field of the system is  the vector field $X_{\overline{H}}:\mathcal{T}\to T\mathcal{T}$   solution of the dynamical system $i_{X_{\overline{H}}}\overline{\Omega}=d\overline{H}$.

Notice that the optimal control problem is obviously regular and thus applying the geometric algorithm of presymplectic systems (see \cite{GotNes:1979,GotNes:1980}) to  $\left(\mathcal{T},\overline{\Omega},\overline{H}\right)$, we obtain a symplectic system on the manifold $W_1=\{(\alpha_x,Y,\xi,U)\in\mathcal{T}:U=X_{\xi}\}$, where $X_{\xi}\in\mathfrak{g}$ is the tangent vector identified with the covector $\xi\in\mathfrak{g}^*$ by the Riemannian metric of $G$. Hence, $(W_1,\overline{H}_{W_1},{\overline{\Omega}}_{W_1})$ is  a symplectic system, with $\overline{\Omega}_{W_1}$  and $\overline{H}_{W_1}$ being the restrictions to $W_1$ of (\ref{eq:twoform}) and (\ref{eq:HamiltonianPresympl}), respectively.
The  map $f$ defined below,  gives us  a diffeomorphism between the symplectic manifolds $\left(T^*(G\times\mathfrak{g}),\Omega_1\right)$ and $\left(W_1,\overline{\Omega}_{W_1}\right)$
\begin{equation} \label{eq:symplectomorphism}
\begin{array}{llll} f: & \left(T^*(G\times\mathfrak{g}),\Omega_1\right) & \longrightarrow & \left(W_1,\overline{\Omega}_{W_1}\right) \\[5pt]
&  (\alpha_x,Y,\xi)  & \longmapsto & (\alpha_x,Y,\xi,X_{\xi}) .\end{array}
\end{equation}
So, we have a symplectomorphism between the two manifolds (see \cite{1978AbrahamMarsden}, p 177). In this sense, we construct the Hamiltonian $H_1:=\overline{H}_{W_1}\circ f:T^*(G\times\mathfrak{g})\to\mathbb{R}$. We get
\begin{equation} \label{eq:HamiltonianH1}
H_1(\alpha_x,Y,\xi)=\left(T^*_eL_x\alpha_x\right)(Y)+\frac{1}{2}\xi\left(X_{\xi}\right),
\end{equation}
for each $(\alpha_x,Y,\xi)\in T^*_{(x,Y)}(G\times\mathfrak{g})$, where $(x,Y)\in G\times\mathfrak{g}$. Furthermore, the existence of the symplectomorphism  (\ref{eq:symplectomorphism}) allows us to conclude that (see \cite{1978AbrahamMarsden}, p 194) the study of the dynamical system defining the Hamiltonian vector field $X_{\overline{H}_{W_1}}$ associated to $\overline{H}_{W_1}$  is reduced to the study of the system
\begin{equation} \label{eq:dynamicalsystem}
i_{X_{H_1}}\Omega_1=dH_1,
\end{equation}
where the vector field $X_{H_1}:T^*(G\times\mathfrak{g})\to T(T^*(G\times\mathfrak{g}))$ is the push-forward of $X_{\overline{H}_{W_1}}$ by $f^{-1}$, $X_{H_1}=(f^{-1})_*X_{\overline{H}_{W_1}}$. The integral curves of this vector field determine the solutions of the optimal control problem  (see \cite{BarMDieLec:2007,DTelIb:2003}).

%%-------------------------------------------------------- Left trivialization of the dynamics
\subsection{Left trivialization of the dynamics} \label{subsection:LeftTrivDyn}

Consider now the left trivialization of the cotangent bundle $T^*\left(G\times\mathfrak{g}\right)$, determined by the diffeomorphism $\rho$ defined from $T^*(G\times\mathfrak{g})$ to the space $G\times\mathfrak{g}\times\mathfrak{g}^*\times\mathfrak{g}^*$ as
\[
\rho(\alpha_x,Y,\xi)=\left(x,Y,T^*_{(e,0)}L_{(x,Y)}(\alpha_x,Y,\xi)\right),
\]
which using (\ref{eq:LeftTranslation}) gives $\rho(\alpha_x,Y,\xi)=\left(x,Y,T^*_eL_x\alpha_x+\mathrm{ad}^*_Y\xi,\xi\right)$, for each $(\alpha_x,Y,\xi)$ in $T_{(x,Y)}^*(G\times\mathfrak{g})$, where $(x,Y)\in G\times\mathfrak{g}$.
Observe that, if $(x,Y,\mu,\xi)\in G\times\mathfrak{g}\times\mathfrak{g}^*\times\mathfrak{g}^*$, then
$\rho^{-1}(x,Y,\mu,\xi)=(T^*_xL_{x^{-1}}\left(\mu-\mathrm{ad}^*_Y\xi\right),Y,\xi)\in T^*_{(x,Y)}(G\times\mathfrak{g})$.

The left trivialization of the Hamiltonian (\ref{eq:HamiltonianH1}) is given by $H:=H_1\circ \rho^{-1}$. We easily conclude that
\begin{equation} \label{eq:HamiltonianTriv}
H(x,Y,\mu,\xi)=\mu(Y)+\frac{1}{2}\xi(X_{\xi}).
\end{equation}
Since $\rho$ is a diffeomorphism,  we can endow (see
\cite{1978AbrahamMarsden}, p 177)
$G\times\mathfrak{g}\times\mathfrak{g}^*\times\mathfrak{g}^*$ with a
symplectic structure, as $\Omega=(\rho^{-1})^*\Omega_1$. Furthermore,
$\rho$ is a symplectomorphism  and  (see \cite{1978AbrahamMarsden}, p
194) the Hamiltonian vector field $X_{H_1}$ defined by
(\ref{eq:dynamicalsystem}) may be left trivialized to
$G\times\mathfrak{g}\times\mathfrak{g}^*\times\mathfrak{g}^*$ by
considering the push-forward by $\rho$ of the Hamiltonian vector field
associated to $H_1$, $X_{H}:=\rho_*X_{H_1}$. %=T\rho\circ X_{H_1}\circ
                                %\rho^{-1}.$ 

The proposition below leads us to the expression of  $X_H$:
\begin{equation} \label{eq:CampoVectHamilSolucao}
X_H(x,Y,\mu,\xi)=(T_eL_xY,X_{\xi}, 0, -\mu+\mathrm{ad}^*_{Y}\xi),
\end{equation}
for each $(x,Y,\mu,\xi)\in G\times\mathfrak{g}\times\mathfrak{g}^*\times\mathfrak{g}^*$.

\begin{prop} \label{prop:HamiltonianEquations}
The following set of differential equations describe the motions of the Hamiltonian system $\left(G\times\mathfrak{g}\times\mathfrak{g}^*\times\mathfrak{g}^*,\Omega,H\right)$
\begin{equation} \label{eq:SistemaHamiltoniano}
\left\{\begin{array}{l} \dot{x}= T_eL_xY \\ \dot{Y}=X_{\xi}
\\ \dot{\mu}=0 \\ \dot{\xi}=-\mu+\mathrm{ad}^*_{Y}\xi .\end{array} \right.
\end{equation}
\end{prop}
\begin{proof}
Let $z=(x,Y,\mu,\xi)$ be an integral curve of $X_H$. Following
(\cite{BatesCushman:1997}, section A.3, example 3) the Hamilton
equations on
$G\times\mathfrak{g}\times\mathfrak{g}^*\times\mathfrak{g}^*$ (the
left trivialization of the cotangent bundle of  the Lie group
$G\times\mathfrak{g}$) are called the Euler-Arnold equations and are
given by 
\[
\fl\left\{ \begin{array}{l} \displaystyle \left(\dot{x},\dot{Y}\right)=T_{(e,0)}L_{(x,Y)}\left(\frac{\partial H}{\partial \mu}(z),\frac{\partial H}{\partial \xi}(z)\right) \\[10pt] \displaystyle \left(\dot{\mu},\dot{\xi}\right)=-T^*_{(e,0)}L_{(x,Y)}\left(\frac{\partial H}{\partial x}(z),\frac{\partial H}{\partial Y}(z)\right)+\mathrm{ad}^*_{\left(\frac{\partial H}{\partial \mu}(z),\frac{\partial H}{\partial \xi}(z)\right)}(\mu,\xi).\end{array}\right.
\]
In this notation, $\partial H(z)/\partial x$  is regarded as an element of $T^*_xG$,  $\partial H(z)/\partial Y$ as an element of  $\mathfrak{g}^*$, $\partial H(z)/\partial \mu$ and $\partial H(z)/\partial \xi$ as  elements of  $\mathfrak{g}$.  Use (\ref{eq:LeftTranslation}) and (\ref{eq:adjunt}) to rewrite  the previous system  as
\[
\fl \left\{\begin{array}{l} \dot{x}= T_eL_x\frac{\partial H}{\partial \mu}(z) \\[5pt] \dot{Y}=\frac{\partial H}{\partial \xi}(z)+\mathrm{ad}_Y\frac{\partial H}{\partial \mu}(z)
 \\[5pt]\dot{\mu}= -T^*_eL_x\frac{\partial H}{\partial x}(z)-\mathrm{ad}^*_Y\frac{\partial H}{\partial Y}(z)+\mathrm{ad}^*_{\frac{\partial H}{\partial \mu}(z)}\mu+\mathrm{ad}^*_{\frac{\partial H}{\partial \xi}(z)}\xi\\[5pt]\dot{\xi}=-\frac{\partial H}{\partial Y}(z)+\mathrm{ad}^*_{\frac{\partial H}{\partial \mu}(z)}\xi.\end{array} \right.
\]
From the expression of the Hamiltonian function (\ref{eq:HamiltonianTriv}) it comes $\partial H(z)/\partial x=0$, $\partial H(z)/\partial Y=\mu$, $\partial H(z)/\partial \mu=Y$ and $\partial H(z)/\partial \xi=X_{\xi}$. Now, substitute these expressions in the above system, use the fact that $\mathrm{ad}^*_{X_{\xi}}\xi=0$ and the result follows. \opensquare
\end{proof}

\begin{remark} \label{remark:EquivEulHamil}
It will now be interesting to see how the dynamics described by (\ref{eq:CampoVectHamilSolucao}) is related with the known variational approach of cubic polynomials. To proceed, we begin with the following remarks:
\begin{itemize}
\item First write the last equation  of  (\ref{eq:SistemaHamiltoniano}) as an equation on the Lie algebra, using the identification of covectors and tangent vectors giving by the Riemannian metric of $G$ (see the end of subsection \ref{subsection:Liegroup} for details on notation). We get $\dot{X}_{\xi}=-X_{\mu}-\mathrm{ad}_YX_{\xi}$.
\item Differentiate the above equation and use the third equation of (\ref{eq:SistemaHamiltoniano}), to obtain $\ddot{X}_{\xi}+\mathrm{ad}_{\dot{Y}}X_{\xi}+\mathrm{ad}_Y\dot{X}_{\xi}=0$. Use the second equation of (\ref{eq:SistemaHamiltoniano}), to get $\tdot{Y}+[Y,\ddot{Y}]=0$.
\end{itemize}
We have just shown that  each solution of the equations of Hamilton (\ref{eq:SistemaHamiltoniano}) gives rise to a  solution of the  equations
\begin{equation}\label{eq:EulerLagrange}
Y=T_xL_{x^{-1}}\dot{x}\qquad\tdot{Y}+[Y,\ddot{Y}]=0.
\end{equation}
Conversely,  solutions of (\ref{eq:EulerLagrange})  satisfying $\dot{Y}=X_{\xi}$ and $\dot{X}_{\xi}+X_{\mu}+\mathrm{ad}_YX_{\xi}=0$, correspond to solutions of (\ref{eq:SistemaHamiltoniano}).

The equations (\ref{eq:EulerLagrange}) are the Euler-Lagrange equations (\ref{eq:cubic}) that define the  cubic polynomials on a Lie group, which were  proved in \cite{CrSLei:1995} as an extension of the proof that had already been given in \cite{NoaHeiPad:1989} for $SO(3)$. (The proof use some facts derived from theorem \ref{theorem:basicpropertiesLieGroup}.)
\end{remark}

%%----------------------------------------------------------------------------------------------------------------------------------------%%
%%-------------------------------------------------------- Reduction ---------------------------------------------------------------------%%
%%----------------------------------------------------------------------------------------------------------------------------------------%%
\section{Reduction of the Hamiltonian system} \label{section:ReductionHamilSystem}

The purpose of this section is to study the  symmetries of the Hamiltonian system $(G\times\mathfrak{g}\times\mathfrak{g}^*\times\mathfrak{g}^*,\Omega,H)$ described in the previous section and use that to reduce the corresponding dynamics, eliminating degrees of freedom in the system. The idea is to apply the symplectic point reduction theorem (see \cite{MarWei:1974} for the original references and \cite{2004OrtegaRatiu} for full details in this subject) and carry out the appropriate interpretation of the reduced Hamiltonian system for the study of important questions as the integrability of the system. Namely, we shall focus our attention on the integrals of motion of the reduced Hamiltonian system.

%%-------------------------------------------------------- Symplectic point reduced space
\subsection{Symplectic point reduced space}

Let $\phi$ be the smooth left action  of the Lie group $G$ on $G\times\mathfrak{g}\times\mathfrak{g}^*\times\mathfrak{g}^*$ defined by
\begin{equation} \label{eq:action}
\phi(g,(x,Y,\mu,\xi))=(gx,Y,\mu,\xi),
\end{equation} for each $g\in G$ and $(x,Y,\mu,\xi)\in G\times\mathfrak{g}\times\mathfrak{g}^*\times\mathfrak{g}^*$.
The moment map of $\phi$ is the map $J:G\times\mathfrak{g}\times\mathfrak{g}^*\times\mathfrak{g}^*\to\mathfrak{g}^*$ defined, for each $(x,Y,\mu,\xi)\in G\times\mathfrak{g}\times\mathfrak{g}^*\times\mathfrak{g}^*$, by
\begin{equation}
J(x,Y,\mu,\xi)=\mathrm{Ad}^*_{x^{-1}}(\mu-\mathrm{ad}^*_Y\xi).
\end{equation}
The action $\phi$ can be seen as the left trivialization to $G\times\mathfrak{g}\times\mathfrak{g}^*\times\mathfrak{g}^*$ of the cotangent lift of the action of $G$ on $G\times\mathfrak{g}$  given, for each $g\in G$ and $(x,Y)\in G\times\mathfrak{g}$, by $(g,(x,Y))\mapsto L_{(g,0)}(x,Y)=(gx,Y)$. Recall that every cotangent lift action is symplectic and has momentum map $\mathrm{Ad}^*$-equivariant (see \cite{1978AbrahamMarsden}, p 283). So, it is easy to verify the following statement:
\begin{description}  \item[(A)] $\phi$ is a symplectic action with momentum map $\mathrm{Ad}^*$-equivariant.
\end{description}
Observe now that the action $\phi$ is proper since it is an action of
a compact Lie group. Moreover, $\phi$ is obviously free and hence the
symmetry algebra of every point in
$G\times\mathfrak{g}\times\mathfrak{g}^*\times\mathfrak{g}^*$ is zero,
which is equivalent to say that every $\eta\in\mathfrak{g}^*$ is a
regular value of the momentum map $J$. 

Let $\eta\in\mathfrak{g}^*$. Consider  the coadjoint isotropy subgroup of  $\eta$, defined by
\begin{equation} \label{eq:Geta}
G_{\eta}:=\{g\in G\,:\,\mathrm{Ad}^*_{g^{-1}}\eta=\eta\}
\end{equation}
and also the level set $J^{-1}(\eta)$ of the momentum map $J$. Note that
\begin{equation} \label{eq:JTrivInversoEta}
J^{-1}(\eta)=\{\left(x,Y,\mu,\xi\right)\in G\times\mathfrak{g}\times\mathfrak{g}^*\times\mathfrak{g}^*\,:\,\mu=\mathrm{Ad}^*_x\eta+\mathrm{ad}_Y^*\xi\}.
\end{equation}
Because $\phi$ is a symplectic $G$-action on the symplectic manifold $G\times\mathfrak{g}\times\mathfrak{g}^*\times\mathfrak{g}^*$ and $\eta\in\mathfrak{g}^*$ is a regular value of $J$, we see that $J^{-1}(\eta)$ is a submanifold of $G\times\mathfrak{g}\times\mathfrak{g}^*\times\mathfrak{g}^*$. Furthermore, as a consequence of $J$ being $\mathrm{Ad}^*$-equivariant, we easily prove that  $J^{-1}(\eta)$ is $G_{\eta}$-invariant. The comments now exposed allow us to conclude that $G_{\eta}$ acts on $J^{-1}(\eta)$ and that the orbit space
\begin{equation} \label{eq:OrbitSpace}
{\left(G\times\mathfrak{g}\times\mathfrak{g}^*\times\mathfrak{g}^*\right)}_{\eta}:=J^{-1}(\eta)/G_{\eta}
\end{equation}
is well defined. The action of $G_{\eta}$ on $J^{-1}(\eta)$ is obtain by restriction of $\phi$ to the subgroup (\ref{eq:Geta}) and to the $G_{\eta}$-invariant submanifold (\ref{eq:JTrivInversoEta}). It turns out that the action $\phi$ is proper and free and that by definition $G_{\eta}$ is a closed  subgroup of $G$,  thus (see \cite{2004OrtegaRatiu}, p 60):
\begin{description}  \item[(B)] The action of $G_{\eta}$ on $J^{-1}(\eta)$ is proper and free.
\end{description}
This result guarantees that the orbit space (\ref{eq:OrbitSpace}) is a smooth manifold and that the corresponding projection map is a surjective submersion.

Since the conditions {\bf (A)} and {\bf (B)} are satisfied, we are able to apply the symplectic point reduction theorem. The theorem states the following:

\begin{quote}
The reduced space ${\left(G\times\mathfrak{g}\times\mathfrak{g}^*\times\mathfrak{g}^*\right)}_{\eta}$ has a unique symplectic structure $\Omega_{\eta}$ characterized by the identity $\pi^*_{\eta}\Omega_{\eta}=i^*_{\eta}\Omega,$ where $i_{\eta}$ is the canonical inclusion from $J^{-1}(\eta)$ to $G\times\mathfrak{g}\times\mathfrak{g}^*\times\mathfrak{g}^*$ and $\pi_{\eta}$ is the projection of $J^{-1}(\eta)$ onto ${\left(G\times\mathfrak{g}\times\mathfrak{g}^*\times\mathfrak{g}^*\right)}_{\eta}$ .
\end{quote}
The symplectic manifold $\left({\left(G\times\mathfrak{g}\times\mathfrak{g}^*\times\mathfrak{g}^*\right)}_{\eta},\Omega_{\eta}\right)$ is called the  {\it symplectic point reduced space}  at $\eta$.

\nosections

Let us now explore in more detail the reduction obtained. More specifically, we will interpret the symplectic point reduced space in a strategic way to conduct further studies. In what follows, we shall adopt the notation $\phi$ for the above $G_{\eta}$-action on $J^{-1}(\eta)$. First, we notice that from (\ref{eq:JTrivInversoEta}), the submanifold $J^{-1}(\eta)$ is diffeomorphic to the semidirect product $G\times\mathfrak{g}\times\mathfrak{g}^*$ (of the Lie group $G\times\mathfrak{g}$ and the vectorial space $\mathfrak{g}^*$) through the diffeomorphism
\begin{equation} \label{eq:diff}
\begin{array}{llll}\Upsilon_{\eta}: & G\times\mathfrak{g}\times\mathfrak{g}^* & \longrightarrow & J^{-1}(\eta)\\[5pt] & (x,Y,\xi) & \longmapsto & (x,Y,\mathrm{Ad}^*_x\eta+\mathrm{ad}^*_Y\xi,\xi).\end{array}
\end{equation}
Consider the $G_{\eta}$-action on $G\times\mathfrak{g}\times\mathfrak{g}^*$ given, for each $g\in G_{\eta}$ and $(x,Y,\xi)\in G\times\mathfrak{g}\times\mathfrak{g}^*$, by $g\cdot (x,Y,\xi)=(gx,Y,\xi)$  and consider also the corresponding orbit space $(G\times\mathfrak{g}\times\mathfrak{g}^*)/G_{\eta}$.
\begin{lemma}
The diffeomorphism (\ref{eq:diff}) induces a new diffeomorphism
\begin{equation}
\bar{\Upsilon}_{\eta}: \left(G\times\mathfrak{g}\times\mathfrak{g}^*\right)/G_{\eta}\longrightarrow (G\times\mathfrak{g}\times\mathfrak{g}^*\times\mathfrak{g}^*)_{\eta},
\end{equation}
which maps the $G_{\eta}$-orbit of the element $(x,Y,\xi)\in G\times\mathfrak{g}\times\mathfrak{g}^*$ into the \mbox{$G_{\eta}$-orbit} of the element $(x,Y,\mathrm{Ad}^*_x\eta+\mathrm{ad}^*_Y\xi,\xi)\in J^{-1}(\eta)$.
\end{lemma}
\begin{proof}
We have just to prove that $\Upsilon_{\eta}$ is equivariant for the  $G_{\eta}$-action $\phi$ on $J^{-1}(\eta)$ and the $G_{\eta}$-action on $G\times\mathfrak{g}\times\mathfrak{g}^*$ described above. Indeed, if $g\in G_{\eta}$ and $(x,Y,\xi)\in G\times\mathfrak{g}\times\mathfrak{g}^*$, then $\Upsilon_{\eta}(g\cdot (x,Y,\xi))=(gx,Y,\mathrm{Ad}^*_{gx}\eta+\mathrm{ad}^*_Y\xi,\xi)=(gx,Y,\mathrm{Ad}^*_x(\mathrm{Ad}^*_g\eta)+\mathrm{ad}^*_Y\xi,\xi) =(gx,Y,\mathrm{Ad}^*_x\eta+\mathrm{ad}^*_Y\xi,\xi)=\phi_g(\Upsilon_{\eta}(x,Y,\xi))$. \opensquare
\end{proof}

\begin{prop}
The point reduced space ${\left(G\times\mathfrak{g}\times\mathfrak{g}^*\times\mathfrak{g}^*\right)}_{\eta}$ is diffeomorphic to  the space $\mathcal{O}_{\eta}\times\mathfrak{g}\times\mathfrak{g}^*$, where $\mathcal{O}_{\eta}$ denotes the coadjoint orbit of the element $\eta$.
\end{prop}
\begin{proof}
It is clear that the map $\varepsilon_{\eta}:\left(G\times\mathfrak{g}\times\mathfrak{g}^*\right)/G_{\eta} \to\mathcal{O}_{\eta}\times\mathfrak{g}\times\mathfrak{g}^*$, which takes a \mbox{$G_{\eta}$-orbit} of an element $(x,Y,\xi)\in G\times\mathfrak{g}\times\mathfrak{g}^*$  to a current point $(\mathrm{Ad}_x^*\eta,Y,\xi)\in \mathcal{O}_{\eta}\times\mathfrak{g}\times\mathfrak{g}^*$, is a diffeomorphism. Hence, one constructs the map $\bar{\varphi}_{\eta}:=\varepsilon_{\eta}\circ \bar{\Upsilon}^{-1}_{\eta}$, that is
\begin{equation} \label{eq:DifeoRedRed}
\begin{array}{llll} \bar{\varphi}_{\eta}: & {\left(G\times\mathfrak{g}\times\mathfrak{g}^*\times\mathfrak{g}^*\right)}_{\eta} & \longrightarrow & \mathcal{O}_{\eta}\times\mathfrak{g}\times\mathfrak{g}^* \\[5pt]  & [(x,Y,\mathrm{Ad}^*_x\eta+\mathrm{ad}^*_Y\xi,\xi)]_{\eta} & \longmapsto & (\mathrm{Ad}^*_x\eta,Y,\xi),\end{array}
\end{equation}
which gives us the diffeomorphism. \opensquare
\end{proof}
The result from the previous proposition allows us to conclude now that $\bar{\varphi}_{\eta}$ is a symplectomorphism, being $\bar{\Omega}_{\eta}={(\bar{\varphi}_{\eta}^{-1})}^*\Omega_{\eta}$ the symplectic structure on $\mathcal{O}_{\eta}\times\mathfrak{g}\times\mathfrak{g}^*$.

To finish, it is useful  to notice  that the map $\varphi_{\eta}:=\bar{\varphi}_{\eta}\circ \pi_{\eta}:J^{-1}(\eta)\to\mathcal{O}_{\eta}\times\mathfrak{g}\times\mathfrak{g}^*$ is such that
\begin{equation}\label{eq:varphiEta}
\varphi_{\eta}(x,Y,\mathrm{Ad}^*_x\eta+\mathrm{ad}^*_Y\xi,\xi)=(\mathrm{Ad}^*_x\eta,Y,\xi).
\end{equation}
Furthermore, $\varphi_{\eta}$ is surjective since $\bar{\varphi}_{\eta}$ is bijective and the projection $\pi_{\eta}$ is surjective.

Besides the reduction of the phase space, the symplectic point reduction theorem  has a  dynamic counterpart, which will be addressed in the next subsection.

%%-------------------------------------------------------- Reduction of the dynamics
\subsection{Reduction of the dynamics}

We proceed with the analysis of the reduction of the dynamics of the Hamiltonian system of cubic polynomials $(G\times\mathfrak{g}\times\mathfrak{g}^*\times\mathfrak{g}^*,\Omega,H)$ described in subsection \ref{subsection:LeftTrivDyn}. We first present the natural reduction of dynamics that comes from the symplectic point reduction theorem. Then we shall perform this reduction  as a dynamics on a Hamiltonian system on  $(\mathcal{O}_{\eta}\times\mathfrak{g}\times\mathfrak{g}^*,\bar{\Omega}_{\eta})$ in the context of the previous subsection, that is, using the diffeomorphism (\ref{eq:DifeoRedRed}).

Consider the Hamiltonian $H$ given by (\ref{eq:HamiltonianTriv}) and the associated Hamiltonian vector field $X_H$ defined by (\ref{eq:CampoVectHamilSolucao}). Notice that $H$ is invariant under the $G$-action defined by (\ref{eq:action}). The symplectic point reduction theorem allows us to conclude the following:
\begin{quote}
The flow $f_t$ of the Hamiltonian vector field $X_H$ induces a flow $f_t^{\eta}$ on the reduced space ${\left(G\times\mathfrak{g}\times\mathfrak{g}^*\times\mathfrak{g}^*\right)}_{\eta}$ defined by $\pi_{\eta}\circ f_t\circ i_{\eta}=f_t^{\eta}\circ\pi_{\eta}$. The vector field generated by the flow $f_t^{\eta}$  is Hamiltonian with associated  reduced Hamiltonian function $H_{\eta}$ defined uniquely by $H_{\eta}\circ \pi_{\eta}=H\circ i_{\eta}$. Furthermore, the vector fields $X_H$ and $X_{H_{\eta}}$ are $\pi_{\eta}$-related.
\end{quote}
The triple $({\left(G\times\mathfrak{g}\times\mathfrak{g}^*\times\mathfrak{g}^*\right)}_{\eta},\Omega_{\eta},H_{\eta})$ is called the {\it reduced Hamiltonian system}. We are interested now in characterizing  the corresponding  system on $(\mathcal{O}_{\eta}\times\mathfrak{g}\times\mathfrak{g}^*,\bar{\Omega}_{\eta})$. Namely, we shall determine the expression of the reduced Hamiltonian vector field when interpreted as a vector field on $\mathcal{O}_{\eta}\times\mathfrak{g}\times\mathfrak{g}^*$ regarding the description given at the end of the previous subsection. To effect this, one follows the steps below.

From now on, wherever there is no confusion, we will denote an element $\mathrm{Ad}^*_x\eta\in\mathcal{O}_{\eta}$ by $\theta$, $\theta:=\mathrm{Ad}^*_x\eta$. Introduce the Hamiltonian function  \mbox{$h:=H_{\eta}\circ\bar{\varphi}_{\eta}^{-1}:\mathcal{O}_{\eta}\times\mathfrak{g}\times\mathfrak{g}^*\to\mathbb{R}$}, with $\bar{\varphi}_{\eta}$ defined by (\ref{eq:DifeoRedRed}).
\begin{lemma}
For each $(\theta,Y,\xi)\in\mathcal{O}_{\eta}\times\mathfrak{g}\times\mathfrak{g}^*$, we have
\begin{equation} \label{eq:HamilReduzReduz}
h(\theta,Y,\xi)=\theta(Y)+\frac{1}{2}\xi\left(X_{\xi}\right).
\end{equation}
\begin{proof}
Since the function $\varphi$ defined by (\ref{eq:varphiEta}) is surjective, we know that the element $(x,Y,\theta+\mathrm{ad}^*_Y\xi,\xi)\in J^{-1}(\eta)$ is such that $\varphi_{\eta}(x,Y,\theta+\mathrm{ad}^*_Y\xi,\xi)=(\theta,Y,\xi)$. Thus,  $h(\theta,Y,\xi)=(H_{\eta}\circ\pi_{\eta})(x,Y,\theta+\mathrm{ad}^*_Y\xi,\xi)$. But from the definition of $H_{\eta}$, we have $H_{\eta}\circ\pi_{\eta}=H\circ i_{\eta}$, so use  (\ref{eq:HamiltonianTriv}) and the result follows. \opensquare
\end{proof}
\end{lemma}

Since $\bar{\varphi}_{\eta}$ is a symplectomorphism, we see that the Hamiltonian vector field $X_h$ associated to $h$ is such that $X_h\circ\bar{\varphi}_{\eta}=T\bar{\varphi}_{\eta}\circ X_{H_{\eta}}$. Then, $X_h\circ \bar{\varphi}_{\eta}\circ\pi_{\eta}=T\bar{\varphi}_{\eta}\circ X_{H_{\eta}}\circ \pi_{\eta}$, that is, $X_h\circ\varphi_{\eta}=T\bar{\varphi}_{\eta}\circ X_{H_{\eta}}\circ\pi_{\eta}$, where $\varphi_{\eta}$ is the function defined by (\ref{eq:varphiEta}). Now, use the fact that $X_H$ and $X_{H_{\eta}}$ are $\pi_{\eta}$-related, that is, $T\pi_{\eta}\circ X_H\circ i_{\eta}=X_{H_{\eta}}\circ\pi_{\eta}$, to obtain
\begin{equation} \label{eq:RelCVreduzreduzCVinicial}
X_h\circ \varphi_{\eta}=T\varphi_{\eta}\circ X_H\circ i_{\eta}.
\end{equation}
We shall develop the expression of $X_h$ for each point in $\mathcal{O}_{\eta}\times\mathfrak{g}\times\mathfrak{g}^*$, using the relation now obtained and after we present a useful remark.

\begin{lemma} \label{lemma:Tzvarphi}
If $(a,b,c,d)\in T_{(x,Y,\theta+\mathrm{ad}^*_Y\xi,\xi)}J^{-1}(\eta)$, then
\begin{equation} \label{eq:lemTzvarphi}
T_{(x,Y,\theta+\mathrm{ad}^*_Y\xi,\xi)}\varphi_{\eta}(a,b,c,d)=\left(\mathrm{ad}^*_{T_xL_{x^{-1}}a}\theta,b,d\right).
\end{equation}
\end{lemma}
\begin{proof}
First we would like to clarify that the choice of the element $(a,b,c,d)$  is related to the fact that $T_{(x,Y,\theta+\mathrm{ad}^*_Y\xi,\xi)}J^{-1}(\eta)$ be a subset of $T_{(x,Y,\theta+\mathrm{ad}^*_Y\xi,\xi)}(G\times\mathfrak{g}\times\mathfrak{g}^*\times\mathfrak{g}^*)=
T_xG\times\mathfrak{g}\times\mathfrak{g}^*\times\mathfrak{g}^*$. Now, let $\beta=(\beta_1,\beta_2,\mathrm{Ad}^*_{\beta_1}\eta+\mathrm{ad}^*_{\beta_2}\beta_3,\beta_3)$  be a curve   in $J^{-1}(\eta)$ satisfying the initial conditions $\beta(0)=(x,Y,\theta+\mathrm{ad}^*_Y\xi,\xi)$ and $\dot{\beta}(0)=(a,b,c,d)$. Then, we know that $T_{(x,Y,\theta+\mathrm{ad}^*_Y\xi,\xi)}\varphi_{\eta}(a,b,c,d)={\mathrm{d}(\varphi_{\eta}\circ \beta)(0)/\mathrm{d}t}$, which is equal to $(\mathrm{ad}^*_{T_{\beta_1(0)}L_{\beta_1(0)^{-1}}\dot{\beta}_1(0)}\mathrm{Ad}^*_{\beta_1(0)}\eta,\dot{\beta_2}(0),\dot{\beta_3}(0))$, that is, $(\mathrm{ad}^*_{T_xL_{x^{-1}}a}\theta,b,d)$ as we wanted to prove. \opensquare
\end{proof}

We are now able to prove the following result.
\begin{prop} \label{prop:CVHamilReduzReduz}
For each point $(\theta,Y,\xi)\in\mathcal{O}_{\eta}\times\mathfrak{g}\times\mathfrak{g}^*$, the dynamical vector field $X_h$ of the Hamiltonian system $(\mathcal{O}_{\eta}\times\mathfrak{g}\times\mathfrak{g}^*,\bar{\Omega}_{\eta} ,h)$ is given by
\begin{equation} \label{eq:CVHamilReduzReduz}
X_h(\theta,Y,\xi)=(\mathrm{ad}^*_Y\theta,X_{\xi},-\theta).
\end{equation}
\end{prop}
\begin{proof}
We begin  by noticing that due to the relations  (\ref{eq:varphiEta}) and (\ref{eq:RelCVreduzreduzCVinicial}), we get $X_h(\theta,Y,\xi)=(T\varphi_{\eta}\circ X_H)(x,Y,\theta+\mathrm{ad}^*_Y\xi,\xi)$. Now  according (\ref{eq:CampoVectHamilSolucao}) we have $X_H(x,Y,\theta+\mathrm{ad}^*_Y\xi,\xi)=(T_eL_xY,X_{\xi},0,-\theta)$. It remains only to show that $T_{(x,Y,\theta+\mathrm{ad}^*_Y\xi,\xi)}\varphi_{\eta}(T_eL_xY,X_{\xi},0,-\theta)=(\mathrm{ad}^*_Y\theta,X_{\xi},-\theta)$, but this follows from (\ref{eq:lemTzvarphi}) taking $(a,b,c,d)=(T_eL_xY,X_{\xi},0,-\theta)$. \opensquare
\end{proof}

Thus, the equations of Hamilton on the reduced manifold  $\mathcal{O}_{\eta}\times\mathfrak{g}\times\mathfrak{g}^*$ are given by
\begin{equation}  \label{eq:EqHamilReduzReduz}
\left\{\begin{array}{l} \dot{\theta}=\mathrm{ad}^*_Y\theta\\ \dot{Y}=X_{\xi}
 \\ \dot{\xi}=-\theta .\end{array} \right.
 \end{equation}

\begin{remark}
In remark \ref{remark:EquivEulHamil}, we have proved the equivalence between the solutions of the equations of Hamilton (\ref{eq:SistemaHamiltoniano}) and the Euler-Lagrange equations (\ref{eq:EulerLagrange}).  It is obvious that the reduced dynamics described by (\ref{eq:CVHamilReduzReduz}) is also related with the variational approach of cubic polynomials. In fact, an integral curve of the reduced Hamiltonian vector field (\ref{eq:CVHamilReduzReduz}) give rise to a curve that satisfies the second equation of the Euler-Lagrange system (\ref{eq:EulerLagrange}). Indeed, writing  the first equation  of  (\ref{eq:EqHamilReduzReduz}) as an equation on the Lie algebra (see the end of subsection \ref{subsection:Liegroup} for details on notation),  we get $\dot{X}_{\theta}+[Y,X_{\theta}]=0$. But, by the other two equations  of (\ref{eq:EqHamilReduzReduz}), we know that $X_{\theta}=-\ddot{Y}$. We conclude that a solution of  the reduced system gives a solution of
\begin{equation}
\tdot{Y}+[Y,\ddot{Y}]=0.
\end{equation}
\end{remark}

%%-------------------------------------------------------- Invariants along the extremal trajectories
\subsection{Invariants along the extremal trajectories}

Integrals of motion of a dynamical system are quantities that are conserved along the flow of that system and can be sometimes associated to symmetries of the system. A classical result due to Liouville, exposed in \cite{Arnold:1989} by Arnold (which has also contributed to a more complete version of this result), says that a dynamical system on a phase space of dimension $2N$ is completely integrable if it admits $N$ (almost everywhere) functionally independent first integrals in involution (i.e., their Poisson brackets all vanish). However, these situations are rather rare. In practice, one often deals with Hamiltonian systems which admits a non-abelian group of symmetries or an abelian group of symmetries in number less than the required  to have complete integrability. If some special conditions are satisfied,  the non-abelian set of independent integrals can lead us to the integrability of the system, as explained by Fomenko and Mishchenko in \cite{MishFomenk:1978}, authors of the theorem of the non-commutative integrability. But, in most cases one naturally expects to find only  a number of independent Poisson commuting invariants less than $N$, which can allow us to partially reduce the original system (Poincar\'e-Lyapunov-Liouville-Arnol'd theorem, \cite{Nekhoroshev:1994}).

The problem we are concerned with in this subsection is the preliminar
analysis of the symmetries of the Hamiltonian system
$(\mathcal{O}_{\eta}\times\mathfrak{g}\times\mathfrak{g}^*,\bar{\Omega}_{\eta}
,h)$, so that we can study the integrability of the system in a
forthcoming work. The problem of reduction of the order of a
Hamiltonian system is an old subject of study, with emphasis on
several works of Poincar\'e and Cartan, namely the Lie-Cartan theorem
(for more details on this subject, see \cite{1988ArnKozNein}). We
shall find, by using this classical theorem of Lie-Cartan, a maximum
number of functionally independent first integrals in involution. 

In the rest of this paper, for the sake of simplicity, we shall use the following notation:
\begin{equation} \label{eq:Notation}
\mathrm{dim}\mathfrak{g}=n \qquad \mbox{and} \qquad\mathrm{dim}\mathcal{O}_{\eta}=2m
\end{equation}
(recall that the dimension of the coadjoint orbit is always even),
where obviously $2m\leq n$. So, the dimension of the phase space of
our Hamiltonian system shall be 
\begin{equation} \label{eq:NotationDimSympleManif}
\mathrm{dim}\left(\mathcal{O}_{\eta}\times\mathfrak{g}\times\mathfrak{g}^*\right)=2(n+m).
\end{equation} 
 
Besides, we shall restrict ourselves to the case of semisimple Lie
groups, for technical reason which will become clear below.

A function
$f:\mathcal{O}_{\eta}\times\mathfrak{g}\times\mathfrak{g}^*\to\mathbb{R}$
is an integral of motion of our dynamical system (with associated
vector field $X_h$) if $[X_h(w)](f)=0$, that is,
$[(\mathrm{d}f)(w)](X_h(w))=0$, for all $w\in
\mathcal{O}_{\eta}\times\mathfrak{g}\times\mathfrak{g}^*$. It is
important to notice that
$\mathrm{d}f:\mathcal{O}_{\eta}\times\mathfrak{g}\times\mathfrak{g}^*\to
T^*(\mathcal{O}_{\eta}\times\mathfrak{g}\times\mathfrak{g}^*)$ is such
that $df(w)\in T^*_w\mathcal{O}_{\eta}\times
\mathfrak{g}^*\times\mathfrak{g}\subseteq\mathfrak{g}\times\mathfrak{g}^*\times\mathfrak{g}$. In
that sense, we shall assume the notation
$\mathrm{d}f(w)=\left( \partial f(w)/\partial\theta,\partial
  f(w)/\partial Y,\partial f(w)/\partial\xi\right)$. 

The Hamiltonian function  is naturally an integral of motion of the Hamiltonian system. So, the function (\ref{eq:HamilReduzReduz}), that is,
\[
l_1\equiv h=\theta(Y)+\frac{1}{2}\xi\left(X_{\xi}\right)
\]
is an integral of motion. But besides that, we are able to prove the following result:
\begin{prop}
The functions $l_{i+1}\,:\mathcal{O}_{\eta}\times\mathfrak{g}\times\mathfrak{g}^*\to\mathbb{R}$ defined by
\begin{equation} \label{eq:InvN}
l_{i+1}=(\theta+\mathrm{ad}^*_Y\xi)(A_i),\;\mbox{with}\;A_i\;\mbox{a fixed basis element of}\;\mathfrak{g},
\end{equation}
are integrals of motion of the Hamiltonian system  $(\mathcal{O}_{\eta}\times\mathfrak{g}\times\mathfrak{g}^*,\bar{\Omega}_{\eta} ,h)$.
\end{prop}
\begin{proof}
Consider $w=(\theta,Y,\xi)\in \mathcal{O}_{\eta}\times\mathfrak{g}\times\mathfrak{g}^*\subseteq\mathfrak{g}^*\times\mathfrak{g}\times\mathfrak{g}^*$, with $\theta=\mathrm{Ad}^*_x\eta$, for some $x\in G$.  An elementary computation gives
\begin{equation} \label{eq:dl1maisi}
\mathrm{d}l_{i+1}(w)=\left(A_i,-\mathrm{ad}^*_{A_i}\xi,\mathrm{ad}_YA_i\right).
\end{equation}
Then, by (\ref{eq:CVHamilReduzReduz}), we get
$[X_h(l_{i+1})](w)=\theta(\mathrm{ad}_YA_i)-\xi(\mathrm{ad}_{A_i}X_{\xi})-\theta(\mathrm{ad}_YA_i)=
\langle[X_{\xi},X_{\xi}],A_i\rangle=0$, which shows that
$[(\mathrm{d}l_{i+1})(w)](X_h(w))=0$, for each $i=1,...,n$, proving
that the given functions are  invariant.  \opensquare
\end{proof}

\begin{remark}
Recall that in the context of the variational approach two invariants are known, (\ref{eq:I1}) and (\ref{eq:I2}). These invariants are related with the invariants now obtained. Indeed, it is simple to prove (using theorem \ref{theorem:basicpropertiesLieGroup}) that  $l_1\equiv I_1$ and $2I_2-\sum_{i=1}^{n}l^2_{i+1}\equiv\theta(X_{\theta})$.
\end{remark}

It is immediate to see that any linear combination of the $n+1$
integrals of motion described  above  is also an integral of motion of
the system. A natural question arises: to extract, from the set of
invariant functions, a maximal set of independent commuting invariant
functions. 

We have the following two results:

\begin{lemma} \label{lemma:Independence}
 If the Lie group $G$ is semisimple, then $\{l_j\}_{ j=1,\ldots, n+1 }$
 is a set of functionally independent functions on  an open
  dense subset of $\mathcal{O}_\eta\times \mathfrak{g}\times
  \mathfrak{g}^*$.
\end{lemma}

\begin{proof}
In the proof  and for the sake of simplicity, we identify $\eta \in
\mathfrak{g}^*$ with an element of $\mathfrak{g}$ via the Riemannian
metric. We shall also consider $\mathcal{O}_{\eta}$ to  be the adjoint orbit defined by a
regular  element $\eta$ in a Cartan subalgebra $\mathfrak{t}$ of $\mathfrak{g}$
and $r$ be the rank of $\mathfrak{g}$.

Consider the orthonormal basis  $\{A_1,...,A_{n}\}$ of the Lie algebra
$\mathfrak{g}$ and represent by $C_{ji}^k$ the structure constants of
this Lie algebra for this basis.  Consider then the coordinate
expression for the invariants:

\[
\begin{array}{l} \displaystyle
 l_1=\sum_{j=1}^ny^j\theta_j(\nu_1,\ldots,\nu_{2m})+\frac{1}{2}\sum_{j=1}^n(\xi_j)^2 \\[5pt]
 \displaystyle     l_{i+1}=\theta_i(\nu_1,\ldots,\nu_{2m})+\sum_{j,k=1}^{n}C_{ji}^ky^j\xi_k, \quad i=1,\ldots,n,
\end{array}\]
where $\nu_1,\ldots,\nu_{2m}$ are the variables in the orbit
$\mathcal{O}_\eta$. The differentials of the invariants can be written
as 
\[
\mathrm{d}l_1=\sum_{\alpha=1}^{2m}\sum_{j=1}^ny^j\frac{\partial
  \theta_j}{\partial\nu_\alpha}d\nu_\alpha+\sum_{j=1}^n\theta_jdy^j+\sum_{j=1}^n
\xi_jd\xi_j 
\]
\[
\mathrm{d}l_{i+1}=\sum_{\alpha=1}^{2m}\frac{\partial \theta_i}{\partial\nu_\alpha}d\nu_\alpha+\sum_{j,k=1}^nC_{ji}^k\xi_kdy^j+\sum_{j,k=1}^nC_{ji}^ky^jd\xi_k, \quad i=1,\ldots,n.
\]
We shall prove that $\mathrm{d}l_1\wedge \mathrm{d}l_2 \wedge \ldots
\wedge \mathrm{d}l_{n+1}\neq 0$ on an open dense subset of
$\mathcal{O}_\eta\times\mathfrak{g}\times\mathfrak{g}^*$. \linebreak
The coefficients of the above exterior product corresponding to the
elements \linebreak $\mathrm{d}\nu_1\wedge \mathrm{d}\nu_2 \wedge
\ldots \wedge \mathrm{d}\nu_{2m}\wedge \mathrm{d}\xi_{i_1} \wedge
\ldots \wedge \mathrm{d}\xi_{i_{r+1}}$  are sums containing $2m$
terms, not depending on the variables $\xi_{i}$, and $r+1$ terms,
each one depending linearly on a different variable $\xi_{i}$. The
$r+1$ terms are given by minors of order $n$ of the matrix
representing the linear map $F$ from $T_{\theta}
\mathcal{O}_{\eta}\times \mathfrak{g}$
 into $\mathfrak{g}$ that applies (Z, W) to $i_{*|\theta}(Z)-ad_YW$,
 where  $i$ is the inclusion of
 $\mathcal{O}_{\eta}$ into $\mathfrak{g}$. If we prove
 that the map $F$ has full rank in an open dense subset of
$\mathcal{O}_\eta\times\mathfrak{g}$, then  the corresponding minor
of order $n$ of the matrix representation gives the non-vanishing term we
are looking for.

In order to do so, let us recall the standard root space decomposition
(see for instance \cite{helgason}) for the complexified algebra
$\mathfrak{g}^{\mathbb{C}}$: 
$$
\mathfrak{g}^{\mathbb{C}}=\mathfrak{g}^{\mathbb{C}}_0\oplus(\bigoplus_{\alpha
  \in \Delta}\mathfrak{g}^{\mathbb{C}}_{\alpha})
$$
with  respect to a Cartan subalgebra  $\mathfrak{t}^{\mathbb{C}}$  
(i.e., $\mathfrak{g}^{\mathbb{C}}_0$
corresponds to the centralizer of  $\mathfrak{t}^{\mathbb{C}}$ in
$\mathfrak{g}^{\mathbb{C}}$ which is equal to
$\mathfrak{t}^{\mathbb{C}}$ if the algebra is semisimple).    The
related vectors  $X_{\alpha}$, $Y_{\alpha} \in \mathfrak{g}$ such 
that $[T,X_{\alpha}]=\alpha(T) Y_{\alpha}$ and 
$[T,Y_{\alpha}]=-\alpha(T) X_{\alpha}$, for all $T\in \mathfrak{t}$ and for each root $\alpha \in \Delta$,
induce the decomposition 
$$
  \ds \mathfrak{g}=\mathfrak{t}\oplus(\sum_{\alpha \in \Delta_+}\mathbb{R}X_{\alpha}\oplus \mathbb{R}Y_{\alpha})
$$

\noindent and give a basis  $B^1_{\mathfrak{g}}$  of $
\mathfrak{g}$.
Let us consider the tangent space $ T_\theta\mathcal{O}_\eta=\{
[\theta, A], A\in \mathfrak{g} \}$, for each $\theta \in
\mathcal{O}_{\eta}$. Using the basis $B_\mathfrak{g}^1$, it is
possible  to check that there exists an open dense subset of  $\mathcal{O}_\eta$ defined by elements $\theta$  such that
$T_\theta\mathcal{O}_\eta\cap \mathfrak{t}=\{0\}$. 
Under this condition,
it is possible to extend a basis $B_{T_{\theta} \mathcal{O}_{\eta}}$
of $T_{\theta} \mathcal{O}_{\eta}$, using a basis of $\mathfrak{t}$,
in order to obtain a basis $B^2_{\mathfrak{g}}$ of $\mathfrak{g}$.
Now, we consider the basis $B_{T_{\theta} \mathcal{O}_{\eta}}\times
B^1_{\mathfrak{g}}$
 of $T_{\theta} \mathcal{O}_{\eta}\times \mathfrak{g}$ and the basis
$B^2_{\mathfrak{g}}$ of $\mathfrak{g}$. It is clear that the matrix
of the map $F$ relatively to these basis has full rank for all
${\theta}\in \mathcal{O}_{\eta}$ such that $T_{\theta}
\mathcal{O}_{\eta}\cap \mathfrak{t}=\{0\}$ and for all
$Y=T+\sum_{\alpha \in \Delta_+}\left(b_{\alpha}X_{\alpha}+
c_{\alpha}Y_{\alpha}\right)$ with no null coefficients $b_{\alpha}$
and $c_{\alpha}$, for each ${\alpha} \in \Delta_+$.  Therefore, we
proved that the map $F$ has full rank in an open dense subset of
${\mathcal O}_{\eta}\times \mathfrak{g}$. This implies that there is
an open dense subset of ${\mathcal O}_{\eta}\times \mathfrak{g} \times
\mathfrak{g}^{*}$ where the functions $\{ l_{1}, \cdots, l_{n+1}\}$
are functionally independent. 
\end{proof}

\begin{lemma} \label{lemma:LieAlgStru}
Considering the Poisson structure on
$\mathcal{O}_\eta\times\mathfrak{g}\times\mathfrak{g}^*$, the set
$\{l_{i+1}\}_{i=1,\ldots,n}$ is endowed with a Lie algebra structure
that makes it isomorphic to the Lie algebra $\mathfrak{g}$. 
\end{lemma}
\begin{proof}
Consider the orthonormal basis  $\{A_1,...,A_{n}\}$ of the Lie algebra
$\mathfrak{g}$ and represent by $C_{ji}^k$ the structure constants of
this Lie algebra for this basis. If
$w=(\theta,Y,\xi)\in\mathcal{O}_{\eta}\times\mathfrak{g}\times\mathfrak{g}^*$,
we know that
$\{l_{i+1},l_{j+1}\}(w)=\left[\mathrm{d}l_{i+1}(w)\right]\left(X_{l_{j+1}}(w)\right)$,
with $X_{l_{j+1}}$ denoting the Hamiltonian vector field associated to
$l_{j+1}$. 

In order to proceed with the proof, let us find  now the expression of
the Hamiltonian vector field $X_{l_{j+1}}$ of $l_{j+1}$  in a similar
way to the one used in proposition  \ref{prop:CVHamilReduzReduz} for
$X_h$. To do this, first consider the function
$L_{j+1}:G\times\mathfrak{g}\times\mathfrak{g}^*\times\mathfrak{g}^*\to\mathbb{R}$
uniquely characterized by the identity $L_{j+1}\circ
i_{\eta}=l_{j+1}\circ\varphi_{\eta}$, where $\varphi_{\eta}$ is the
surjective function defined by (\ref{eq:varphiEta}). More precisely,
we have
$L_{j+1}(z)=l_{j+1}(\theta,Y,\xi)_{|_{\theta=\mu-\mathrm{ad}^*_Y\xi}}=\mu(A_j)$
where the element $z$ is  such that $w=\varphi_{\eta}(z)$, that is,
$z=(x,Y,\theta+\mathrm{ad}^*_Y\xi,\xi)\in J^{-1}(\eta)$. Then, the
Hamiltonian vector field $X_{L_{j+1}}$ associated to the function
$L_{i+1}$  is 
related with $X_{l_{j+1}}$ as follows
\[
\fl X_{l_{j+1}}(w)=T_z\varphi_{\eta}(X_{L_{j+1}}(z))=\left(\mathrm{ad}^*_{T_xL_{x^{-1}}X^1_{L_{j+1}}}\theta,X^2_{L_{j+1}},X^4_{L_{j+1}}\right),
\]
where we use  lemma \ref{lemma:Tzvarphi} choosing
$(a,b,c,d)=X_{L_{j+1}}(z)=(X^1_{L_{j+1}},X^2_{L_{j+1}},X^3_{L_{j+1}},X^4_{L_{j+1}})$. To
completely determine $X_{l_{j+1}}(w)$, we must to find the expression
of the components $X^1_{L_{j+1}}$, $X^2_{L_{j+1}}$ and
$X^4_{L_{j+1}}$. A computation analogous to the one done for the
components of the vector field $X_H$ in proposition
\ref{prop:HamiltonianEquations} shows that $X_{L_{j+1}}^1=T_eL_xA_j$,
$X^2_{L_{j+1}}=\mathrm{ad}_YA_j$ and
$X^4_{L_{j+1}}=\mathrm{ad}^*_{A_j}\xi,$ and hence 
\begin{equation}
X_{l_{j+1}}(w)=\left(\mathrm{ad}^*_{A_j}\theta,\mathrm{ad}_YA_j,\mathrm{ad}^*_{A_j}\xi\right).
\end{equation}
Now, recalling the expression (\ref{eq:dl1maisi}) of $\mathrm{d}l_{i+1}(w)$, we get
\[
\begin{array}{l}
    \fl \{l_{i+1},l_{j+1}\}(w)=\theta\left([A_j,A_i]\right)+\xi\left([A_i,[A_j,Y]]+[A_j,[Y,A_i]]\right)=\\[5pt]
    \fl =\theta([A_j,A_i])-\xi([Y,[A_i,A_j]]) = (\theta+ad^*_Y\xi)([A_j,A_i])= \\[5pt]
    \fl =\displaystyle \sum_{i,j,k=1}^{n}C_{ji}^kl_{k+1}.
\end{array}
\]
The structure constants  of the Lie algebra generated by the functions $l_{i+1},\;i=1,...,n$ and the structure constants of $\mathfrak{g}$ coincide, so the algebras are isomorphic. \opensquare
\end{proof}

Let us summarize the situation for the case of semisimple Lie groups:
\begin{itemize}
\item[--] We have $n+1$ smooth functions, the integrals of motion
  $l_1,l_2,...,l_{n+1}$, whose differentials are linearly independent
  on $\mathcal{O}_{\eta}\times\mathfrak{g}\times\mathfrak{g}^*$ (see
  lemma \ref{lemma:Independence}). 
\item[--] The linear span of these functions is closed with respect to the Poisson bracket (see lemma \ref{lemma:LieAlgStru} and recall that $l_1$ is in involution with all the other functions).
\end{itemize}
Thus, the linear span $\mathfrak{L}$ of these $n+1$ functions has a structure of a finite-dimensional real Lie algebra, with $\mathrm{dim}\mathfrak{L}=n+1$. This algebra is called the algebra of integrals.

We now present the Lie-Cartan theorem formulated according
\cite{1988ArnKozNein}:
% It should be noted that this theorem is a generalization, to the case of  the algebras of \mbox{non-commutative} integrals, of the Poincar\'e method concerning the reduction of the order of a system, a method which is a Hamiltonian version of the reduction process of Routh.
\begin{theorem}[S. Lie - E. Cartan]
Consider a Hamiltonian system $(M,\omega,H)$ with first integrals $F_1,...,F_k$ such that $\{F_i,F_j\}=a_{ij}(F_1,...,F_k)$. Let $F:M\to \mathbb{R}^k$ be the natural mapping generated by these set of integrals.

Suppose that the point $c\in\mathbb{R}^k$ is not a critical value of the mapping $F$ and that in its neighborhood the rank of matrix $(a_{ij})$ is constant. Then in a small neighborhood $U\subset\mathbb{R}^k$ of $c$ one can find $k$ independent functions $\varphi_j:U\to\mathbb{R}$ such that the functions $\phi_j=\varphi_j\circ F:N\to\mathbb{R}$, where $N=F^{-1}(U)$, satisfy the relations
\[
\{\phi_1,\phi_2\}=...=\{\phi_{2q-1},\phi_{2q}\}=1,
\]
whereas the remaining brackets $\{\phi_i,\phi_j\}$ vanish. The number $2q$ equals the rank of the matrix $(a_{ij})$.
\end{theorem}
We are interested in the following consequence of the above theorem (see \cite{1988ArnKozNein}):
\begin{remark} \label{remark:LieCartan}
Under the hypotheses of Lie-Cartan theorem and using the
  notation above, there are $k-q$ independent integrals in involution:
  $\phi_2, \phi_4,..., \phi_{2q-2}, \phi_{2q},
  \phi_{2q+1},...,\phi_{k}$. As a consequence, the  original
  Hamiltonian system can be reduced, by the method of Poincar\'e, to a
  system with minus $k-q$ degrees of freedom than the original one. 
\end{remark} 

Thus, if we consider the case of a semisimple group $G$, and an open dense subset $D$ in
$\mathcal{O}_{\eta}\times\mathfrak{g}\times\mathfrak{g}^*$, the
functional independence of the $n+1$ integrals is satisfied and
the skew-symmetric Poisson bracket matrix $\left(\{l_i,l_j\}\right)$,
$1\leq i,j\leq n+1$, has maximal rank. Notice that, by lemma
\ref{lemma:LieAlgStru}, the maximum rank of this Poisson bracket
matrix  coincides with the maximum rank of the matrix
$M_{\mathfrak{g}}(a)=\left(M_{ij}(a)\right)$, with
$M_{ij}(a)=\sum_{k=1}^{n}C^k_{ij}a_{k}$, for $a=(a_1,...,a_n)\in
\mathbb{R}^n$ and where $C^k_{ij}$ are structure constants of the Lie
algebra $\mathfrak{g}$. We fix the notation (note that this rank is
always even):
\begin{equation} \label{eq:rankMg}
\mathrm{r}_{\mathfrak{g}}:=\frac{1}{2}\mathrm{max}_{a\in \mathbb{R}^n}\mathrm{rank}M_{\mathfrak{g}}(a).
\end{equation}

In consideration of  remark \ref{remark:LieCartan}, the
$2(m+n)$-dimensional Hamiltonian system
$(\mathcal{O}_{\eta}\times\mathfrak{g}\times\mathfrak{g}^*,\bar{\Omega}_{\eta},h)$
admits $ n+1-\mathrm{r}_{\mathfrak{g}}$ functions  defined on an open
subset of $D$, which form a set of independent integrals of the motion
in involution. Thus, we can expect the system to be reduced to a
system of dimension equal to $2(m+\mathrm{r}_{\mathfrak{g}}-1)$.

\begin{example}
Consider the problem of cubic polynomials on the Lie group  $SO(3)$, which can be
illustrated by the dynamic optimal control problem of the spherical
free rigid body study by the authors of this paper in
\cite{AbrCamCGal:2010,AbrCamCGal:2011}. In this case, it is well known that  the
coadjoint orbit $\mathcal{O}_{\eta}$ corresponds to a $2$-dimensional
sphere with radius $\|\eta\|$. (For the singular case $\eta=0$  the
orbit reduces to one point.) So, considering the non-singular case,
the symplectic reduced manifold
$\mathcal{O}_{\eta}\times\mathfrak{so}(3)\times\mathfrak{so}^*(3)$ has
dimension equal to $8$ and $\mathrm{r}_{\mathfrak{g}}=1$. Applying the
above theorem, the corresponding reduced Hamiltonian system has  three
independent invariants in involution and it can be reduced to a
$2$-dimensional system.
\end{example}

As mentioned at the beginning of this section, these results on
integrals of the motion may have important implications
from the point of view of the integrability of the corresponding
dynamical systems.  Thus, they are relevant  at the level of
determining the cubic polynomials.  In conclusion, by using the
Lie-Cartan theorem, we are able to reduce the Hamiltonian system to a
system with at least, $n+1-\mathrm{r}_{\mathfrak{g}}$ degrees of
freedom less. 
%We
%note that if $m+\mathrm{r}_{\mathfrak{g}}=1$, the system will be
%completely integrable. But, this condition happens only in trivial
%cases, when the algebra $\mathfrak{g}$ is abelian or the coadjoint
%rbit $\mathcal{O}_{\eta}$  is reduced to a point.

In a future paper we hope to address this integrability problem in detail.

%%----------------------------------------------------------------------------------------------------------------------------------------%%
%%-------------------------------------------------------- References ---------------------------------------------------------------------%%
%%----------------------------------------------------------------------------------------------------------------------------------------%%

\end{document}